\documentclass{amsart}

\usepackage{latexsym}
\usepackage{amssymb}
\usepackage{graphics}
\usepackage{graphicx}
\usepackage{epstopdf}

\newtheorem{theorem}{Theorem}[section]
\newtheorem{proposition}[theorem]{Proposition}
\newtheorem*{theorem*}{Theorem}

\theoremstyle{definition}

\newtheorem{remark}[theorem]{Remark}
\newtheorem*{remark*}{Remark}

\newcommand{\tr}{\operatorname{tr}}
\newcommand{\trace}{\operatorname{trace}}
\newcommand{\PGL}{\operatorname{PGL}}
\newcommand{\PSL}{\operatorname{PSL}}
\newcommand{\GL}{\operatorname{GL}}
\newcommand{\SL}{\operatorname{SL}}
\newcommand{\Isom}{\operatorname{Isom}}

\newcommand{\C}{\mathbb C}
\newcommand{\R}{\mathbb R}
\newcommand{\Q}{\mathbb Q}
\newcommand{\HH}{\mathbb H}
\everymath{\displaystyle}

\usepackage{setspace}

\begin{document}

\title[Intercusp geodesics]{Intercusp geodesics and the invariant trace field of hyperbolic
  3-manifolds} 

\author{Walter D Neumann} \address{Department of Mathematics, Barnard
  College, Columbia University, 2990 Broadway MC4429, New York, NY
  10027, USA} \email{neumann@math.columbia.edu} 

\author{Anastasiia Tsvietkova} \address{Department of Mathematics,
  University of California - Davis, One Shields Ave, Davis, CA 95616, USA}
\email{tsvietkova@math.ucdavis.edu}



\begin{abstract}Given a cusped hyperbolic 3-manifold with finite
  volume, we define two types of complex parameters which capture geometric
  information about the preimages of geodesic arcs traveling
  between cusp cross-sections. We prove that these parameters are elements
  of the invariant trace field of the manifold, providing a connection between the intrinsic geometry of a 3-manifold and its number-theoretic invariants. Further, we explore the question of choosing a minimal collection of arcs and associated parameters to generate the field. We prove that for a tunnel number $k$ manifold it is enough to choose $3k$ specific parameters. For many hyperbolic link complements, this approach allows one to compute the field from a link diagram. We also give examples of infinite families of links where a single parameter can be chosen to generate the field, and the polynomial for it can be constructed from the link diagram as well. 

\end{abstract}
%\doublespacing

\maketitle

\section{Introduction}
The invariant trace field is one of the most used tools in the study
of hyperbolic manifolds from the number-theoretical point of view. In
this note, we discuss how this arithmetic invariant is related to the
intrinsic geometry of the manifold, and to intercusp geodesics in
particular. This geometric perspective allows one to compute the
invariant trace field of many hyperbolic link complements from their
diagrams.

$M$ will always denote a complete orientable hyperbolic $3$-manifold
of finite volume. If $\Gamma$ is the image of a discrete faithful
representation of the fundamental group of $M$ into $\Isom^+(\HH^3)$,
then $M$ can be regarded as the quotient $\HH^3/\Gamma$. After picking
upper half space coordinates $\C\times\R_+$ on $\HH^3$ we can identify
$\Isom^+(\HH^3)$ with $\PSL_2(\C)$ acting so that its action on the
sphere at infinity $\C\cup\{\infty\}$ is by M\"obius
transformations. This identification is only determined up to
conjugacy, since it depends on the coordinate choice. Nevertheless,
for an element $\gamma\in \Gamma$ we
can speak of the \emph{trace} $\tr(\gamma)$ (determined only up to
sign), since the trace of a matrix is
invariant under conjugacy.

The field $\tr(\Gamma)$ generated by the traces of elements of
$\Gamma$ is called the \emph{trace field} of $M$. In view of the
Mostow-Prasad rigidity, $\tr(\Gamma)$ is a finite extension of $\Q$
(the proof can be found in \cite{Maclachlan}). It is an invariant of
the group $\Gamma$ and thus is a topological invariant of the
manifold, but in general it is not an invariant of its
commensurability class in $\PSL_2(\C)$ (see \cite{Borel, NeumannReid,
  Reid} for counterexamples).

Consider the subgroup $\Gamma^2=\langle\gamma^2 ~|~ \gamma\in
\Gamma\rangle$ of $\Gamma$. The \emph{invariant trace field} is the field
generated  over $\Q$ by the traces of $\Gamma^2$.  Often
denoted by $k(\Gamma)$ or $k(M)$, the invariant trace field is a
topological and commensurability invariant of the manifold
(\cite{Reid}). Clearly, it is a subfield of the trace field. If $M$ is
a link complement, it actually coincides with the trace field (this
was proven in \cite{Reid} for knots, and in \cite{NeumannReid} for
links).

In \cite{NeumannReid}, it is shown that the invariant trace field
contains useful geometric information about the hyperbolic
manifold. In particular, if $\Gamma$ contains parabolic elements, the
invariant trace field is equal to the field generated by shapes of all
tetrahedra of any ideal triangulation of $M$ (by ``shape'' of an ideal
tetrahedron we mean the cross-ratio of the vertices as elements of
$\C\cup\{\infty\}$; it is determined up to a three-fold ambiguity,
depending on an orientation-compatible choice of ordering of the
vertices). {Here we show that the invariant trace field also contains  certain
complex ``intercusp parameters'' that
measure distances and angles between cusps (more precisely, 
between preimages in $\HH^3$ of cusp cross-sections), as well as
``translation parameters'' which measure displacement between ends of
intercusp geodesic arcs.

With the ideas described in \cite{ThistlethwaiteTsvietkova} this
enables one to compute generators of the invariant trace field of a
hyperbolic link directly from a link diagram in many cases. For
example, for $2$-bridge links, we demonstrate that a single intercusp
parameter suffices, with a polynomial which can be constructed
combinatorially from the diagram.  Previously known methods included
finding decimal approximations of simplex shapes, and then making an
intelligent guess of the corresponding polynomial for the field using
the LLL algorithm (see \cite{Neumann2000}).}

\section{The parameters}\label{sec:params}
In this section, we assume our hyperbolic $3$-manifold $M$ has at
least one cusp. We introduce a complex parameter that captures
geometric information about distances and angles between preimages of
cusp cross-sections in $M$. Later we will use it to compute invariant
trace fields of links from their diagrams. The idea of such a
parameter appeared in \cite{ThistlethwaiteTsvietkova} for intercusp
geodesics that correspond to crossings of a link diagram. Here we will
consider it in a more general setting. Our definition also dovetails
with a notion of ``complex length" of a geodesic that was introduced
in \cite{NeumannReid}.

% We identify the boundary of $\HH^3$ with the Riemann sphere
% $\C\cup\{\infty\}$, using the upper half-space model of hyperbolic
% space. 

We will speak loosely of the sphere at infinity as the ``boundary'' of
$\HH^3$.  Each horospherical cusp cross-section of $M$ is a torus
which lifts to a set of horospheres tangent to the boundary of
$\HH^3$. The point of tangency of such a horosphere $H_i$ will be
referred to as the center of $H_i$ and will be denoted by $P_i$.

For each cusp of $M$ we choose an essential simple closed curve in a
horospherical section of the cusp, which we call the \emph{meridian}.
(If $M$ is the complement of an oriented link in $S^3$ we choose the
standard meridians, which are oriented using the right hand screw
rule.) Henceforth we will assume that the horospherical torus
cross-section of each cusp of $M$ has been chosen so that the
(geodesic) meridian curve on this torus has length $1$. Such choice
guarantees that the horoballs have disjoint interiors. The horoballs
are in fact disjoint in every case except for the figure-eight knot
complement in $S^3$ (as was proved \cite{Adams2002}). For the
figure-eight, the corresponding cross-sectional torus touches itself
in two points.

In the following we only consider horospheres in $\HH^3$ which are
lifts of cusp cross-sections of $M$ as above. Each such horosphere can
be regarded as the complex plane, with coordinates specified (up to
translation) by declaring that the meridional translation corresponds
to the real number one.  If we position $H_i$ to be the Euclidean
plane $z=1$ centered at $\infty$ (which we denote by $H_{\infty}$),
% Any length measured on this particular horosphere is simultaneously a
% Euclidean and hyperbolic length.
then the meridional translation on $H_i$ is represented by the
matrix
$\Small\setlength{\arraycolsep}{3pt}\begin{pmatrix}1&1\\0&1\end{pmatrix}$. For
convenience we will often abuse the distinction between
$\Isom^+(\HH^3)=\PSL_2(\C)$ and $\SL_2(\mathbb C)$ and simply work
with matrices. Since $\PSL_2(\C)=\PGL_2(\C)$ we sometimes use matrices
in $\GL_2(\C)$ (but note that the $\PSL$ trace $\tr(A)$ of such a matrix $A$ is
$\frac{\pm\trace(A)}{\sqrt{\det(A)}}$).

Let $\gamma(H_1,H_2)$ be the shortest geodesic arc connecting two
horospheres $H_1$ and $H_2$ in $\HH^3$. If it has length $d$, we can
parallel translate along $\gamma(H_1,H_2)$ and then rotate by an angle
$\theta$ in $H_2$ to take the meridional direction on $H_1$ to the the
meridional direction on $H_2$. We call the complex
number $$\delta(H_1,H_2) :=d+i\theta$$ the \emph{complex distance
  between $H_1$ and $H_2$} and we call
$$ w(H_1, H_2) :=e^{-\delta(H_1,H_2)}$$
the \emph{intercusp parameter}\footnote{In
  \cite{ThistlethwaiteTsvietkova} certain intercusp parameters, with
  opposite sign, are called ``crossing labels''.}.
Fig.~\ref{fig:intercusp} illustrates a complex intercusp distance with
argument $\theta$ between $0$ and $\pi$.
 \begin{figure}[ht]
  \centering
  \includegraphics[scale=0.41]{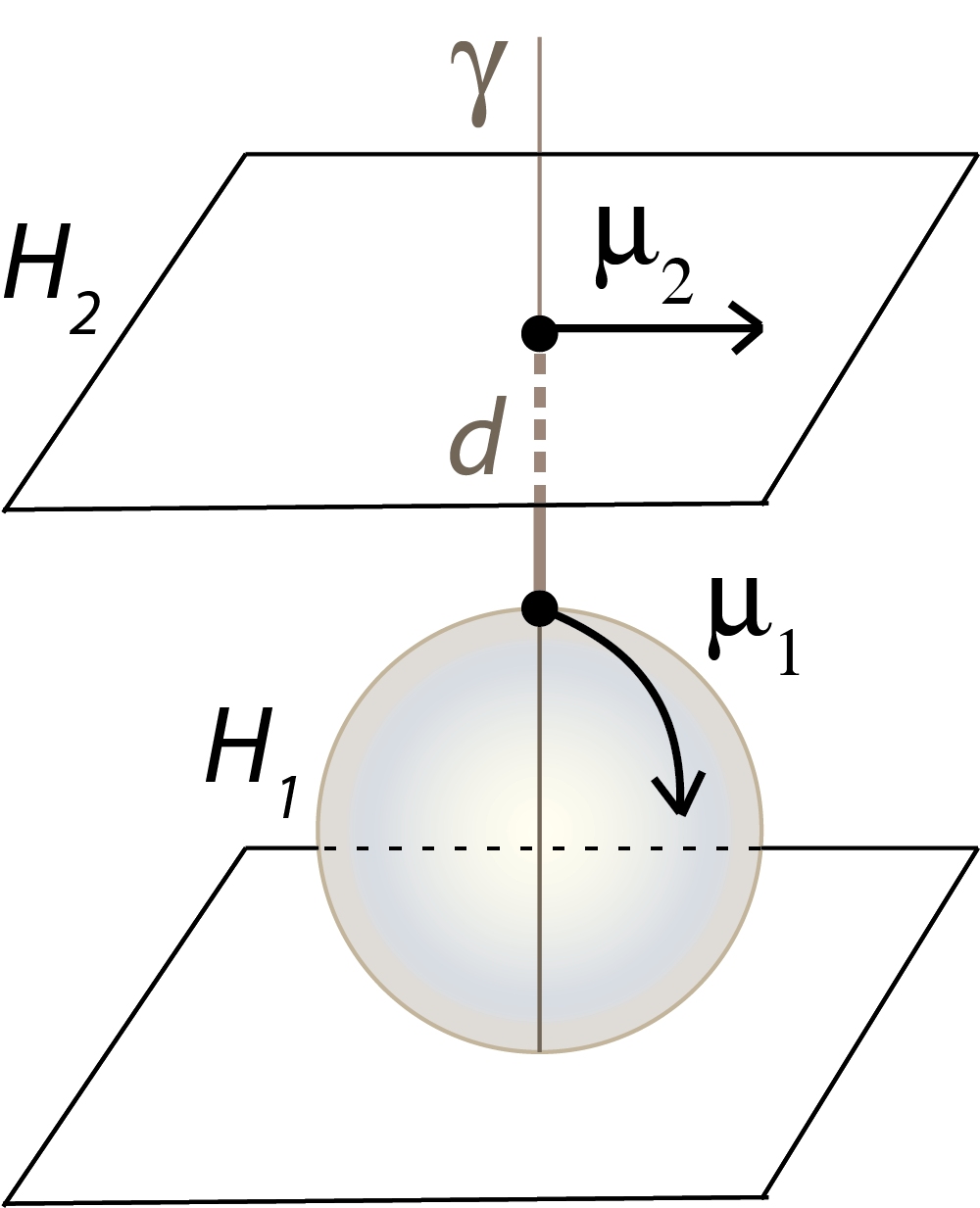} \hspace{0.65 in}
  \caption{Intercusp distance}
  \label{fig:intercusp}
\end{figure}

%  $w(H_1, H_2)$ corresponding
% to horospheres $H_1$, $H_2$ as follows. Let $\gamma$ be a geodesic
% connecting centers of horospheres $H_1$ and $H_2$. Fig.~1 illustrates
% this for the case when one of the horospheres is $H_\infty$. Let the
% modulus of $w(H_1, H_2)$ be $e^{-d}$, where $d$ is the hyperbolic
% length of the part of $\gamma$ between $H_1$ and $H_2$.

Observe that if we position one of $H_1$ and $H_2$ as $H_\infty$ and
the other with center at $0$ then the matrix
\begin{equation}
  \label{eq:intercusp}
  %\Small\setlength{\arraycolsep}{3pt}
M(H_1,H_2):=\begin{pmatrix}0&w(H_1,H_2)\\1&0\end{pmatrix}\in
\PGL(2,\C)
\end{equation}
 exchanges $H_1$ and $H_2$ taking meridian direction of
$H_1$ to that of $H_2$.

We will also use another complex parameter defined as
follows. Suppose we have three horospheres 
%$H_i , i = 1 , 2 , 3$ with 
$H_1\ne H_2\ne H_3$, and $P_i$ is the center of a horosphere $H_i$ for
$i=1,2,3$. Suppose $H_2$ intersects the geodesics $P_1P_2$ and
$P_2P_3$ in points $N$ and $M$ respectively (see
Fig.~\ref{fig:auxiliary}). Using the affine complex structure on $H_2$
there is a complex number determining a translation mapping $N$ to
$M$.  We call this complex number \emph{the translation
  parameter}\footnote{In \cite{ThistlethwaiteTsvietkova} certain
  translation parameters, sometimes with sign changed, are called ``edge
  labels''.}  $u(H_1, H_2, H_3)$. Note that if we position $H_2$ as
$H_{\infty}$ then $u(H_1, H_2, H_3)$ will be the complex number that
corresponds to the translation between the centers $P_1,P_3$ of
$H_1,H_3$. Then the matrix
\begin{equation}
  \label{eq:auxiliary}
  %\Small\setlength{\arraycolsep}{3pt}
M(H_1,H_2,H_3):=\begin{pmatrix}1&u(H_1,H_2,H_3)\\0&1\end{pmatrix}\,.
\end{equation}
gives a translation of $H_2$ taking $N$ to $M$.
\begin{figure}[ht]
  \centering
\includegraphics[scale=0.55]{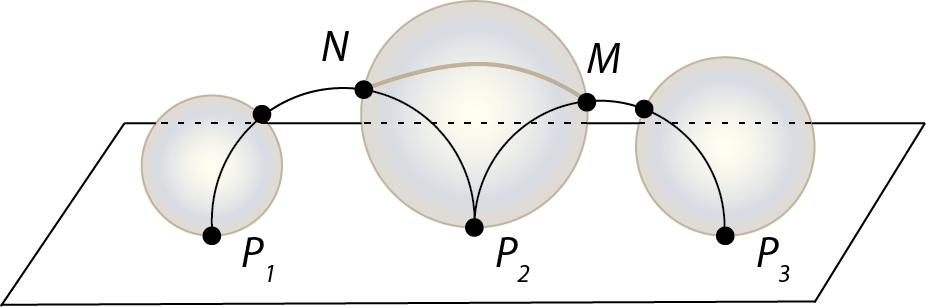} \hspace{0.4 in}
  \caption{Auxiliary parameter}
  \label{fig:auxiliary}
\end{figure}

\section{Properties of the parameters}

We first give versions of Theorems 4.1 and 4.2 of
\cite{ThistlethwaiteTsvietkova} adapted to the orientation conventions
of this paper, along with quick proofs of both.  Let $H_1, \dots, H_n$
be horospheres with $H_i\ne H_{i+1}$ for $i=1,\dots,n-1$ and $H_n\ne
H_1$.  We use the notation of equations \eqref{eq:intercusp} and
\eqref{eq:auxiliary},
\begin{theorem}[Compare Theorem 4.2 of \cite{ThistlethwaiteTsvietkova}]\label{th:cycle relation}
With indices taken modulo $n$,
  $$\prod_{i=1}^n M(H_i,H_{i+1})M(H_i,H_{i+1},H_{i+2})=I\quad
\text{in }\PGL(2,\C)\,.$$
\end{theorem}
\begin{theorem}[Compare Theorem 4.1 of \cite{ThistlethwaiteTsvietkova}]\label{th:shape parameter} Recall $P_i$ denotes the center of $H_i$.
  The shape parameter of the ideal simplex with vertices
  $P_{i-1},P_i,P_{i+1},P_{i+2}$ is
$$\frac{-w(H_i,H_{i+1})}{u(H_{i-1},H_i,H_{i+1})u(H_i,H_{i+1},H_{i+2})}\,.$$  
\end{theorem}
As in  \cite{ThistlethwaiteTsvietkova} we use the shape parameter given
by the cross ratio
$$\zeta :=\frac{(P_{i-1}-P_i)(P_{i+1}-P_{i+2})}{(P_{i-1}-P_{i+1})(P_{i}-P_{i+2})}\,,$$
which gives the parameter associated to the edge $P_iP_{i+1}$. In the
literature this $\zeta$ is often associated with the vertex ordering
$P_i,P_{i+1},P_{i+2},P_{i-1}$.
\begin{proof}[Proof of Theorem \ref{th:cycle relation}]
  Start with $H_1$ positioned as $H_\infty$ and $H_2$ centered at
  $0$. Apply the isometry given by $M(H_1,H_2)$ which exchanges $H_1$
  and $H_2$, so now $H_2$ is positioned at $H_\infty$ with $H_1$
  centered at $0$. Next apply $M(H_1,H_2,H_3)^{-1}$ which translates
  $H_2$ to move $H_3$ to have center $0$. So now $H_2$ is positioned
  as $H_\infty$ and $H_3$ centered at $0$. Now repeat with
  $M(H_2,H_3)$ followed by $M(H_2,H_3,H_4)^{-1}$ to get $H_3$
  positioned as $H_\infty$ and $H_4$ centered at $0$. After $n$ such
  steps we are back to the original positioning, so
$$ M(H_n,H_1,H_2)^{-1}M(H_n,H_1)\dots
M(H_1,H_2,H_3)^{-1}M(H_1,H_2)=I\,.$$
Taking inverse of this equation (and keeping in mind that
$M(H_i,H_j)$ is an involution) gives the desired result.
\end{proof}
\begin{proof}[Proof of Theorem \ref{th:shape parameter}]Recall that
  indices are modulo $n$.  We take $i=0$ and set $w=w(H_1,H_2)$,
  $u'=u(H_0,H_1,H_2)$ , $u=u(H_1,H_2,H_3)$, and we start with $H_1$
  positioned as $H_\infty$ and $H_2$ centered at $0$ as before. Then
  $H_0$ is centered at $-u'$. Apply $M(H_1,H_2)$.  Then the center of
  $H_0$ has been moved to $-w/u'$, $H_1$ is centered at $0$ and $H_2$
  is positioned as $H_\infty$.  Now $H_3$ is centered at $u$. Taking the
  cross-ratio of $P_0=-w/u', P_1=0, P_2=\infty, P_3=u$ gives the
  result.
\end{proof}
\begin{theorem}\label{th:2.1}
  The intercusp parameters $w(H_i,H_j)$ and the translation parameters
  $u(H_i,H_j,H_k)$ lie in the invariant trace field $k(M)$.
\end{theorem}
\begin{proof}
  Let $\mathcal P$ be the set of all centers of horospheres which lift
  from cusp cross-sections of $M$. In \cite[Theorem 2.4]{NeumannReid}
  it is shown that if three points of $\mathcal P$ are positioned at
  $0$, $1$ and $\infty$ then $\mathcal P$ is positioned as subset of 
  $k(M)\cup\{\infty\}\subset \C\cup\{\infty\}$.

  To see $u(H_i,H_j,H_k)\in k(M)$ we position $H_j$ as $H_\infty$ and
  $H_i$ with center at $0$. Then there is also a horosphere centered
  at $1$ so $\mathcal P$ is positioned as a subset of
  $k(M)\cup\{\infty\}$. So
  $u(H_i,H_j,H_k)=P_k-P_i=P_k$ is in $k(M)$.

Since shape parameters are also in $k(M)$, it now follows from Theorem
  \ref{th:shape parameter} applied to a simplex with vertices $P_h,
  P_i, P_j, P_k$ that $w(H_i,H_j)$ is in $k(M)$. 
\end{proof}
If the image $\gamma_{ij}$ in $M$ of an intercusp geodesic arc
$\gamma(H_i,H_j)$ is embedded we call $\gamma_{ij}$ an
\emph{intercusp arc} of $M$ and if the line segment in an $H_j$
joining the endpoints of a $\gamma(H_i,H_j)$ and a $\gamma(H_j,H_k)$
has embedded image $\gamma_{ijk}$ in $M$ we call $\gamma_{ijk}$ a
\emph{cusp arc} of $M$.

\begin{theorem}\label{th:2.2} 
  Suppose $X\subset M$ is a union of cusp arcs and pairwise disjoint
  intercusp arcs, where any intercusp arcs which are not disjoint 
have been bent 
    slightly near intersection points to make them disjoint,
  and suppose $\pi_1(X)\to \pi_1(M)$ is surjective
  (equivalently, the lift $\tilde X\in \HH^3$ is connected). Then the
  intercusp and translation parameters corresponding to these arcs
  generate the invariant trace field.
\end{theorem}

\begin{proof}
If $k$ is a field then the square of an
element
 $\Small\setlength{\arraycolsep}{3pt}\begin{pmatrix}a&b\\c&d
\end{pmatrix}
\in \PGL_2(k)$ 
equals
$\Small\frac1{ad-bc}\setlength{\arraycolsep}{3pt}\begin{pmatrix}a&b\\c&d
\end{pmatrix}^2\in
\PSL_2(k)$ and hence has $\PSL$ trace in $k$.

The conditions on $X$ imply that each covering
transformation in $\Gamma$ of the covering map $\HH^3\to M$ is a
product of matrices of the form $M(H_i,H_j)$ or $M(H_i,H_j,H_k)$. It
is therefore in $PGL_2(k(M))$, so its square has PSL trace in
$k(M)$. By \cite[Theorem 2.1]{NeumannReid} the traces of squares of
elements of $\Gamma$ generate the invariant trace field.
\end{proof}

\newpage

\section{Geometric applications}

\subsection{Zickert's truncated triangulations}\label{subsec:zickert}
In \cite{Zickert} Christian Zickert considers an ideal triangulation
of $M$ with the simplices truncated by removing horoballs centered at
the vertices of the ideal simplices. He allows horoballs of any size,
but we will use the ones normalized as in Section \ref{sec:params}.  He
uses a labelling of these truncated simplices to give a particularly
simple computation of the extended Bloch class and complex volume of
$M$. The label $g_{ij}$ on a long edge of a truncated tetrahedron of
the triangulation 
(see Fig.~\ref{fig:Z}) is a matrix of the form
$\Small\setlength{\arraycolsep}{3pt}\begin{pmatrix}
  0&-\alpha^{-1}\\\alpha&0
\end{pmatrix}$ 
and a label $\alpha_{ik}^j$ on a short edge has the form
$\Small\setlength{\arraycolsep}{3pt}\begin{pmatrix}
  1&u\\0&1
\end{pmatrix}$. 
Note that 
$\Small\setlength{\arraycolsep}{3pt}\begin{pmatrix}
  0&-\alpha^{-1}\\\alpha&0
\end{pmatrix}$ 
can be written as 
$\Small\setlength{\arraycolsep}{3pt}\begin{pmatrix}
  0&-\alpha^{-2}\\1&0
\end{pmatrix}\in\PGL_2(\C)$. 
It is not hard to check that $-\alpha^2$ represents the intercusp
parameter $w(H_i,H_j)$ and $u$ represents the translation parameter
$u(H_i,H_j,H_k)$, so with the horoballs chosen as in this paper,
Zickert's parameters give elements of the form $\alpha^2$ and $u$ in
$k(M)$.

\begin{figure}[h]
\centering
{\includegraphics[width=5.9cm]{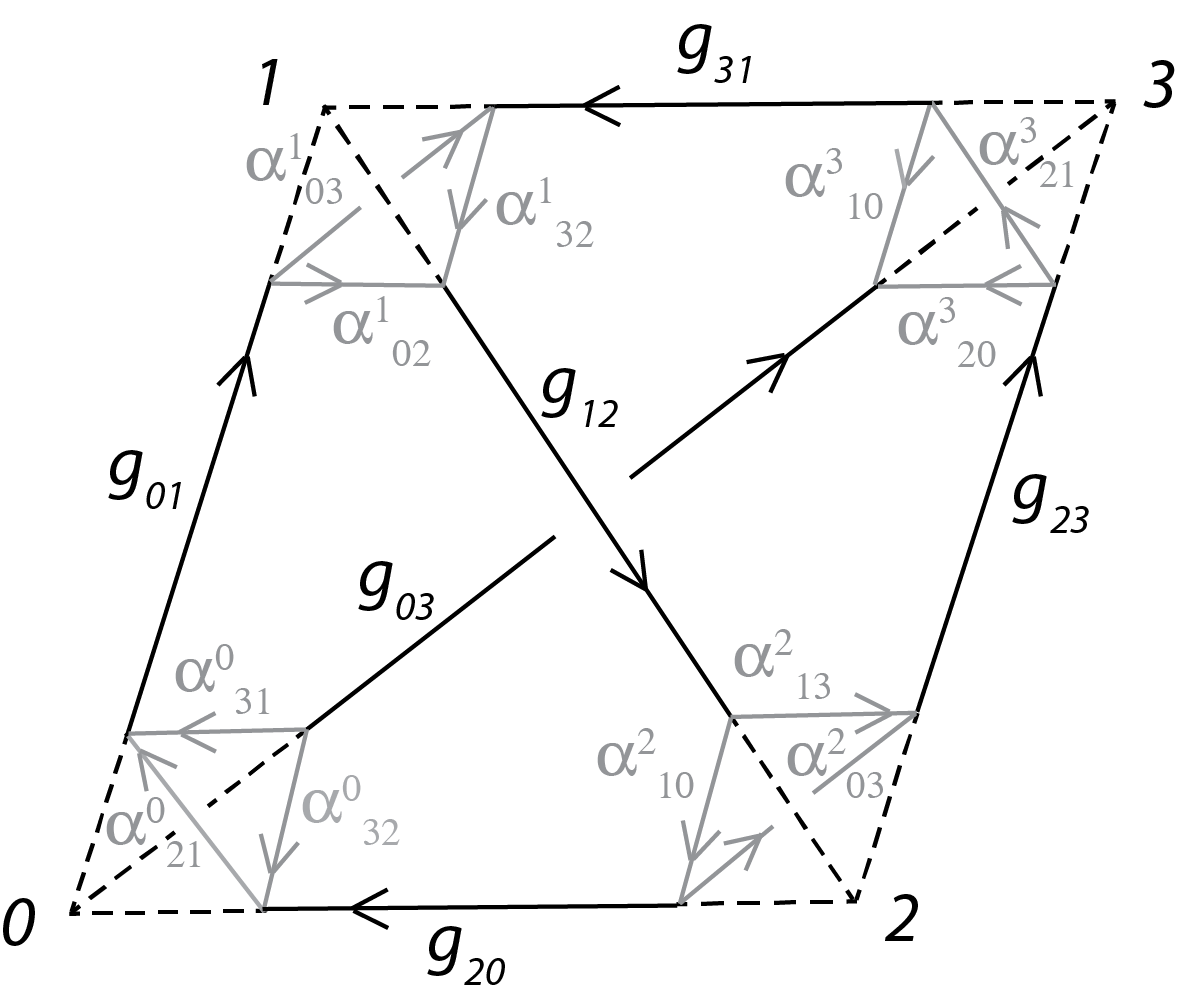}}
\caption{Zickert's truncated tetrahedron with labels}
\label{fig:Z}
\end{figure}

\subsection{Parametrizing hyperbolic structure of link
    complements by complex labels}\label{subsec:TT}

In \cite{ThistlethwaiteTsvietkova}, a new method for computing
hyperbolic structure of links is suggested. It parametrizes horoball
structure using complex labels, which then can be found from a link
diagram that satisfies a few mild restrictions. The method is based on
ideal polygons corresponding to the regions of a link diagram rather
than decomposition of the complement into ideal
tetrahedra. We will
proceed by defining the labels; the method is described after that.
%in detail in Section \ref{sec:4}.

Suppose that $M$ is a link complement and that the link has a reduced
diagram $D$ such that every arc from an overpass to an underpass
  of a crossing is properly homotopic to a geodesic in $M$.
Conjecturally,
every hyperbolic link admits such a diagram; for now it has been
proved that every hyperbolic alternating link does (see
\cite{ThistlethwaiteTsvietkova} for a discussion and for the
sufficient conditions on checkerboard surfaces). Existence of such a
diagram for a link guarantees the applicability of the method.

The boundary of a $k$-sided region $R$ of the diagram $D$ is a union
of $k$ arcs on the boundary torus (we call them \emph{edges of $R$})
and $k$ arcs, each of which goes from an overpass to an underpass of a
crossing. Suppose $\Pi_R$ is a preimage of $R$ in $\mathbb{H}^3$. Then $\Pi_R$ is a cyclic sequence of $k$ segments of geodesics connecting the ideal points $P_1 , \dots , P_{k}$ in $\mathbb{
H}^3$, and $k$ Euclidean segments on the corresponding horospheres $H_1, \dots , H_k$ (Fig.~\ref{fig:RP}).

\begin{figure}[h]
\centering
{\includegraphics[width=7.1cm]{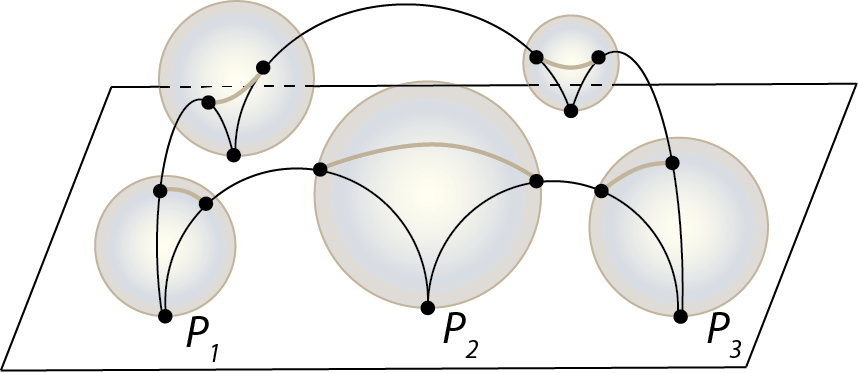}}
\caption{A preimage of the boundary of a 5-sided region of a link diagram}
\label{fig:RP}
\end{figure}

Each geodesic $P_iP_{i+1}$ meets $H_i$ and $H_{i+1}$ in points
$M_i$ and $N_{i+1}$ respectively.  An orientation of the link
determines a direction of the corresponding translation along the Euclidean line segment on $H_i$ in $\mathbb{H}^3$
  joining $M_i$ with
$N_i$. The corresponding translation parameter ($u(H_{i-1}, H_i,
H_{i+1})$ or $u(H_{i+1}, H_i, H_{i-1})=-u(H_{i-1}, H_i,
H_{i+1})$, depending on the orientation)
is called an \emph{edge label} in \cite{ThistlethwaiteTsvietkova}
and is affixed to the side of the
corresponding edge of $R$.  A \emph{crossing label}, affixed to the
crossing arc (or just to the corresponding crossing) that lifts to the
geodesic $P_iP_{i+1}$, is the negative of the intercusp parameter
$w(H_i, H_{i+1})$.

A set of equations for edge and crossing labels,  used to compute the hyperbolic structure on the link complement, is
  given in \cite{ThistlethwaiteTsvietkova}. They consist of three equations resulting from the matrix relation equivalent to the one of Theorem \ref{th:cycle relation}
  for each $\Pi_R$, and, for each arc between two crossings of the
  link diagram, an equation relating the values of the two edge labels
  corresponding to the regions on the two sides of the arc (for an
  alternating link the the equation just says that the two labels
  differ by 1).

%The equations for edge and crossing labels, alternative to the gluing and completeness equations, are derived in \cite{ThistlethwaiteTsvietkova}. The approach utilizes isometries of $\Pi_R$, and the resulting equations can be written just by looking at a reduced link diagram.

 %The negative sign comes from the different conventions for viewing the horospherical torus (in \cite{ThistlethwaiteTsvietkova}, the viewer is situated between the cusps, and thus in a situation shown at Fig.2 one of the meridians is oriented towards the viewer, while the other - away from the viewer).

%Note that given a polyhedral decomposition $\mathcal{T}$ of a hyperbolic 3-manifold $M$ with all vertices ideal, we can similarly introduce labels of its long and short edges. By \cite{Epstein}, at least one decomposition of $M$ into ideal tetrahedra exists. The long edges (in the terminology of \cite{Menasco}) correspond to geodesics between horospheres centered at ideal vertices, and the short ones correspond to edges of Eucledean polygons cut off by the horosphere sections. Given the long edge between the horospheres $H_i , H_{i+1}$, we can assign the intercusp parameter $w(H_i, H_{i+1})$ to it. For every short edge, we first have to choose an orientation. Then we can assign the auxiliary complex parameter to it that corresponds to the Euclidean translation along the edge on the corresponding cutting horosphere. These labels in general are not related to a link diagram anymore. This is reminiscent of the labeling scheme for a triangulation of a hyperbolic 3-manifold suggested in \cite{Zickert} .

To summarize, instead of the traditional gluing and
  completeness relations based on shape parameters for an ideal
  triangulation, %of the link complement, 
we now use two types of complex parameters: the intercusp
parameter $w(H_i, H_{i+1})$ describing %that contains the information about
distance and angle between two chosen cusps
and the translation parameter $u(H_{j-1}, H_j, H_{j+1})$ describing
%that has information about 
how the horospheres are situated with respect to
each other. Theorem \ref{th:2.1} showed that these 
parameters are elements of the invariant trace field. We have a
  finite number of such parameters describing the geometric
structure of $M$, either using
 the labels that are assigned to a link
diagram (edge and crossing labels) or labels
assigned to a polyhedral decomposition (Zickert's parameters).

\subsection{Generating the invariant trace field}
It is not hard to see that Theorem \ref{th:2.2} applies in both the
above cases, so we have:
\begin{proposition}\label{prop:4.1}
  The collection of parameters described above generates the invariant
  trace field in the situations of subsections \ref{subsec:zickert} and
  \ref{subsec:TT}\qed
\end{proposition}

{The number of parameters needed in the above proposition can be
reduced as follows. A collection of intercusp arcs in $M$ is a
\emph{tunnel collection} if the arcs  can be isotoped fixing their
endpoints so that they are disjoint and the  result of removing open horoball
neighborhoods of the cusps and tubular neigborhoods of the 
arcs is a handlebody. The collections of intercusp arcs used in
Proposition \ref{prop:4.1} is a tunnel collection, but usually a small
subset of these intercusp arcs already is. For example, any two-bridge link has a tunnel
collection consisting of a single intercusp arc.
\begin{proposition}
  If $M$ has a tunnel collection consisting of $k$ intercusp arcs,
  then the invariant trace field can be generated by the $k$ intercusp
  parameters of these arcs together with $2k$ translation parameters.
\end{proposition}
\begin{proof} We number the cusps with $i=1,\dots,h$.  Suppose the
  $i$-th horospherical cusp section has $s_i$ endpoints of tunnel arcs
  on it.  We can find a collection of $s_i+1$ cusp arcs connecting
  these endpoints such that their union is a graph whose complement on
  the cusp section is an open disc (as in Fig. \ref{fig:h}). The union of all these cusp arcs
  and isotoped tunnel arcs then satisfies Theorem \ref{th:2.2}.

  The total number of these cusp arcs is $2k+h$. But the $s_i+1$
  translation parameters at the $i$-th cusp section are linearly
  dependent modulo $1$, since a meridian of the cusp section has
  translation parameter $1$. We therefore only need $s_i$ translation
  parameters at the $i$-th cusp in applying Theorem \ref{th:2.2}, for
  a total of $2k$ translation parameters
\end{proof}}

\begin{figure}[t]
\centering
{\includegraphics[width=5.1cm]{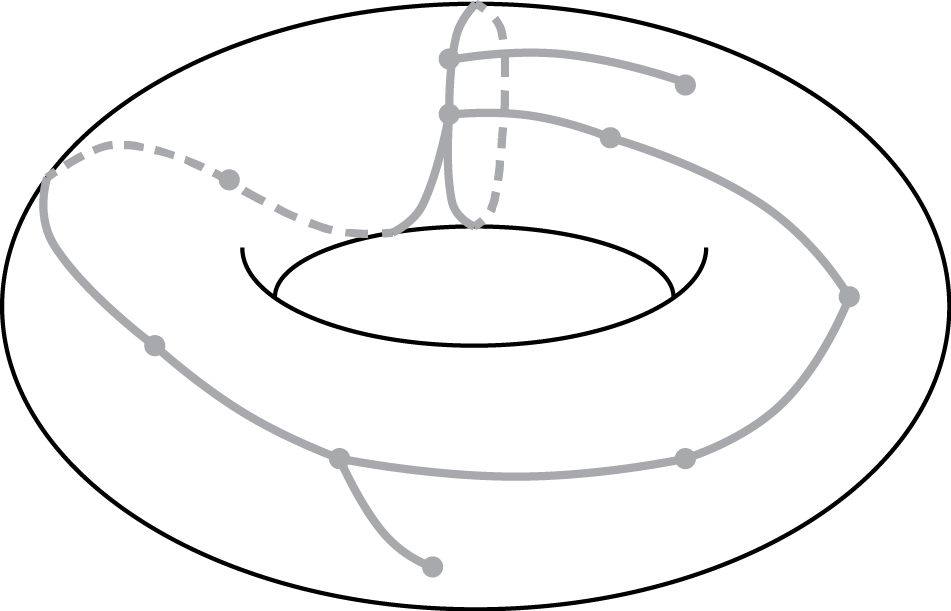}}
\caption{An example of a graph whose complement on
  the cusp section is an open disc}
\label{fig:h}
\end{figure}

In fact, usually a small subset even of the above reduced collection
of labels generates the invariant trace field. For example, in
\cite{Volume} it is shown that for a 2-bridge link there is an ideal
triangulation (in fact the canonical ideal triangulation) for which
the simplex parameters are all rational functions of the crossing
label $w_1$ of the leftmost crossing in the standard alternating
diagram for the link {(this crossing arc is a tunnel)}. By
\cite[Theorem 2.4]{NeumannReid} the simplex parameters of an ideal
triangulation always generate the invariant trace field, so we get:
\begin{proposition} For a two-bridge link the invariant trace field is
  generated by the single crossing label $w_1$ described above.\qed
\end{proposition}

This proposition implies that for a hyperbolic 2-bridge link, the
polynomial $P$ in $w_1$, obtained by applying the recursive process
described in \cite{Volume}, has a zero
which generates the invariant trace field. 

Note that it is not guaranteed that $P$ is irreducible.  In fact,
suppose a reduced alternating diagram of a hyperbolic two-bridge link
has $k$ twists with $n_1 , n_2 , \dots , n_k$ crossings. The
calculation of \cite{Volume} gives an upper bound $m_1^3m_2^3\dots
m_k^3$ for the degree of the polynomial $P$, where $m_i=n_i$ if
$n_i>1$, and $m_i=2$ otherwise.
The work of Riley \cite{Riley} provides a sharper upper bound
$(\alpha-1)/2$ for the degree of the invariant trace field, where
$(\alpha, \beta)$ denotes the normal form of the 2-bridge type, given
by $\alpha/\beta=m_1+1/(m_2+1/(\dots+1/m_k))..)$. Experiment suggests
that Riley's bound is usually sharp.

\begin{remark} 
% For a link diagram with $n$ crossings, there are a
%   priori $3n$ parameters generating the invariant trace field but in
%   practice this number can be reduced considerably. For example, for
  Another example is the infinite family of links that are closures of
  the braid $(\sigma_1\sigma_2^{-1})^n$. Symmetry allows to use just
  three diagram labels, and a quick computation then shows that
  {just one translation parameter suffices and} the
  invariant trace field is generated over $\mathbb{Q}$ by
  $\sqrt{-3-4\cos(\pi/n)+4\cos^2(\pi/n)}$ %and $\cos(\pi/n)$
  (see the ``Examples" section of \cite{ThistlethwaiteTsvietkova} for
  a picture and relations).
\end{remark}

In fact, for a ``random'' knot chosen from existing knot tables it is
rare that the invariant trace field is not generated by a single one
of the parameters, since it is unusual that the invariant trace
field has a proper subfield of degree $>1$.

\subsection*{Acknowledgments:}

The authors would like to thank Morwen Thistlethwaite for help with
computational aspects, and to Alan Reid for bringing attention to
the rep-polynomial in the work of Riley. The authors also acknowledge
support from U.S. National Science Foundation grants DMS 1107452,
1107263, 1107367 ``RNMS: GEometric structures And Representation
varieties" (the GEAR Network).

\end{document}